\documentclass[11pt]{article}
\usepackage[utf8]{inputenc}
\usepackage[margin=1in]{geometry}

\usepackage{hyperref}
\usepackage{authblk}
\usepackage{amsmath}
\usepackage{amssymb}
\usepackage{amsthm}
\usepackage{tikz-cd}
\usepackage{enumerate}

\newcommand*{\NN}{\mathbb{N}}
\newcommand*{\ZZ}{\mathbb{Z}}
\DeclareMathOperator{\im}{im}

\DeclareMathOperator{\id}{id}
\DeclareMathOperator{\Hom}{Hom}
\DeclareMathOperator{\Nat}{Nat}
\DeclareMathOperator{\Tor}{Tor}
\DeclareMathOperator{\Ext}{Ext}
\DeclareMathOperator{\Soc}{soc}
\DeclareMathOperator{\Topp}{top}
\DeclareMathOperator{\spec}{spec}
\newcommand*{\grRMod}{\text{gr}R\text{-}\mb{Mod}}
\newcommand*{\PMod}{\mb{PMod}}

\newcommand*{\Vect}{\mathbf{Vect}}
\newcommand*{\mb}[1]{\mathbf{#1}}
\newcommand*{\mf}[1]{\mathfrak{#1}}
\newcommand*{\abs}[1]{\left\lvert#1\right\rvert}
\newcommand*{\kk}{\mathbf{k}}

\numberwithin{equation}{subsection}
\newtheorem{thm}[equation]{Theorem}
\newtheorem{prop}[equation]{Proposition}
\newtheorem{lemma}[equation]{Lemma}
\newtheorem{cor}[equation]{Corollary}

\theoremstyle{definition}
\newtheorem{defn}[equation]{Definition}
\newtheorem{ex}[equation]{Example}

\theoremstyle{remark}
\newtheorem{rem}[equation]{Remark}

\title{Flat covers and injective hulls of persistence modules}
\author{Eero Hyry\thanks{eero.hyry@tuni.fi}}
\author{Ville Puuska\thanks{puuskaville@gmail.com}}
\affil{Tampere University}

\begin{document}

\maketitle

\begin{abstract}
Motivated by recent progress in topological data analysis, we establish a Matlis duality between injective hulls and flat covers of persistence modules. This extends to a duality between minimal flat and minimal injective resolutions. We utilize the theory of flat cotorsion modules and flat covers developed by Enochs and Xu. By means of this theory we can work with persistence modules which are not tame or even pointwise finite-dimensional.

\end{abstract}


\section*{Introduction}
\label{sec: Introduction}
This article is motivated by recent progress in topological data analysis. Topological data analysis is a recent field of mathematics, which aims to study the shape of data. One of the main methods of topological data analysis is persistent homology. In persistent homology, one first constructs a filtration of a given topological space\textemdash an increasing family of subspaces, indexed by a poset. The homological properties which “persist” along the filtration are considered to be important.  By taking (co)homology with coefficients in a field \(\kk\), one obtains a diagram of vector spaces and linear maps that is called a (co)persistence module. Note that homology and cohomology are dual to each other as \(\kk\)-vector spaces. More formally, a persistence module is a covariant and a copersistence module a contravariant functor from the poset interpreted as a category to the category of \(\kk\)-vector spaces. 

In the following we will assume that the poset is \(\ZZ^n\). From the point of view of graded algebra a persistence module \(M\) then corresponds to a \(\ZZ^n\)-graded module over the the polynomial ring \(\kk[x_1,\ldots, x_n]\). We also identify copersistence modules with persistence modules obtained by "flipping degrees", i.e. a copersistence module \(N\) is identified with the \(\ZZ^n\)-graded module \(\bigoplus_{\mb a \in \ZZ^n} N_{-\mb a}\). Under this identification, the persistent cohomology of a filtered space becomes equal to the Matlis dual of its persistent homology: if we denote the persistent homology by \(M\), then the persistent cohomology is
\[M^\vee = \bigoplus_{\mb a \in \ZZ^n} \Hom_\kk(M_{-\mb a}, \kk).\]

In the case \(n = 1\) every finitely generated persistence module can be uniquely written as a direct sum of the so-called “interval modules”. The ends of an interval describe the ends of a “bar”, and can be seen as the “birth” and “death” of some topological feature. These bars yield the “barcode” of the persistence module, which is a complete and discrete invariant. Unfortunately, this does not apply to multipersistence: if \(n > 1\), an invariant like the barcode does not exist. A promising direction of research to get around the lack of the barcode has been to study minimal presentations and resolutions; typically free ones. A minimal free resolution exists if \(M\) is finitely generated, which is the case in many applications of persistent homology. However, in topological data analysis it is important to be able to consider non-finitely generated persistence modules, too. Note in particular that the Matlis dual of a finitely generated module need not be finitely generated.

The need to deal with non-finitely generated persistence modules lead Miller to consider flat covers and flat resolutions instead of the free ones. He also realized that deaths should be viewed as dual to births (see \cite[Section 1.4]{miller20b}): births correspond to flat covers and deaths to injective hulls. It is therefore important to consider injective hulls and injective resolutions, too. Injective resolutions have recently been proven to be useful beyond mere theory in \cite{bauer23} where dualities between minimal free and injective resolutions and between persistent homology and cohomology were leveraged to develop an algorithm for computing minimal free resolutions of persistent homology.

In this article, we expand Miller’s idea to better understand flat covers and injective hulls of persistence modules, including persistence modules which are not tame or even pointwise finite-dimensional. General flat persistence modules and flat covers of arbitrary persistence modules are sometimes too difficult to work with. These difficulties can often be avoided by assuming that the persistence modules in question are cotorsion. To this purpose, we adapt the theory of flat cotorsion modules and flat covers developed by Enochs and Xu \cite{enochs84, xu96, enochs97} to the setting of persistence modules. Importantly, Matlis duals of persistence modules are cotorsion. This means in particular that persistent cohomology is always cotorsion. Moreover, all pointwise finite-dimensional persistence modules are cotorsion.

The theory of flat cotorsion modules can be considered as a sort of dual to the theory of injective modules. One key example of this duality is Enochs’ decomposition theorem of flat cotorsion modules over a Noetherian ring \cite[Theorem]{enochs84}. This theorem shows that flat cotorsion modules have unique decompositions similar to the unique decompositions of injective modules discovered by Matlis. We prove a persistence module version of the decomposition theorem in Theorem \ref{thm: flat cotorsion structure}. The minimal flat resolution of a cotorsion persistence module consists of flat cotorsion modules so this decomposition theorem gives us an interpretable structure for the minimal flat resolution.

The main result of this article is a Matlis duality between injective hulls and flat covers, Theorem \ref{thm: matlis duality of injective hulls and flat covers}, which states that a morphism \(f \colon M \to N\) is an injective hull, if and only if its Matlis dual \(f^\vee \colon N^\vee \to M^\vee\) is a flat cover. Dually, if \(M\) is assumed to be pointwise finite-dimensional, then \(g \colon M \to N\) is a flat cover, if and only if \(g^\vee\) is an injective hull. As a corollary we obtain a similar duality between minimal flat resolutions and minimal injective resolutions. In the 
single parameter case Theorem \ref{thm: matlis duality of injective hulls and flat covers} generalizes the fact that the barcodes of persistent homology and cohomology are equal, since the barcode is equivalent to the composition of the flat cover and the injective hull. This composition is called a flange presentation in \cite{miller20b}.

\section*{Acknowledgements}
\label{sec: Acknowledgements}
The second author was supported by grants from the Jenny and Antti Wihuri Foundation, the Alfred Kordelin Foundation, and the Vilho, Yrjö and Kalle Väisälä Foundation.

\section{Preliminaries}
\label{sec: Preliminaries}
Throughout this paper, we fix \(n \in \NN\), a field \(\kk\), and \(R = \kk[x_1, \dots, x_n]\). We consider \(R\) to be \(\ZZ^n\)-graded with \(\deg x_i = \mb e_i\) for each \(i\), where \(\mb e_1, \dots, \mb e_n \in \ZZ^n\) is the standard basis.

A \emph{persistence module} is a \(\ZZ^n\)-graded \(R\)-module. For a persistence module \(M\) and \(\mb a \in \ZZ^n\) we will denote the vector space of degree \(\mb a\) homogeneous elements of \(M\) by \(M_\mb a\). We use the notation \(M(\mb a)\) to denote the \emph{degree \(\mb a\) shift of \(M\)}, i.e.~\(M(\mb a) = \bigoplus_{\mb d \in \ZZ^n} M_{\mb a+\mb d}\).

A morphism of persistence modules is an \(R\)-morphism \(f \colon M \to N\) that preserves degrees, i.e.~\(f(M_\mb d) \subseteq N_\mb d\) for all \(\mb d \in \ZZ^n\). We denote the set of morphisms from \(M\) to \(N\) by \(\Nat(M,N)\) and we denote the category of persistence modules by \(\PMod\).

It is well known that persistence modules are equivalently functors \(\ZZ^n \to \Vect\), where \(\Vect\) denotes the category of \(\kk\)-vector spaces. We will use this equivalence throughout, but by a persistence module we will always refer to a \(\ZZ^n\)-graded \(R\)-module.

For a family of persistence modules \((M_i)_{i \in \Lambda}\), the direct sum and direct product are given by
\[\bigoplus_{i \in \Lambda} M_i = \bigoplus_{\mb a \in \ZZ^n} \bigoplus_{i \in \Lambda} (M_i)_\mb a \text{ and } \prod_{i \in \Lambda} M_i = \bigoplus_{\mb a \in \ZZ^n} \prod_{i \in \Lambda} (M_i)_\mb a\]
respectively. Note that while the direct sum of persistence modules is equal to the direct sum of non-graded \(R\)-modules, the direct product of persistence modules can be a strict subset of the direct product of non-graded \(R\)-modules. If the family is constant, i.e.~\(M_i = M\) for all \(i \in \Lambda\), then we denote
\[M^{(X)} := \bigoplus_{i \in \Lambda} M \text{ and } M^X := \prod_{i \in \Lambda} M.\]

For persistence modules \(M\) and \(N\), the tensor product of them as non-graded \(R\)-modules is naturally graded by setting \(\deg m \otimes n = \deg m + \deg n\) for homogeneous elements \(m \in M\) and \(n \in N\). We define the \emph{graded \(\Hom\)} between \(M\) and \(N\) to be the persistence module
\[\underline{\Hom}(M, N) := \bigoplus_{\mb a \in \ZZ^n} \Nat(M, N(\mb a)).\]
Note that a morphism \(f \colon M \to N(\mb a)\) is equvalently an \(R\)-morphism \(M \to N\) such that \(f(M_\mb d) \subseteq N_{\mb a+\mb d}\) for all \(\mb d \in \ZZ^n\). With this we see that \(\underline{\Hom}(M, N)\) is the submodule of the \(R\)-module \(\Hom(M, N)\) generated by all such \(R\)-morphisms.

The usual adjunction between \(\otimes\) and \(\Hom\) induces natural isomorphisms
\[\underline{\Hom}(M \otimes N, L) \cong \underline{\Hom}(M, \underline{\Hom}(N, L))\]
for all persistence modules \(M\), \(N\), and \(L\).

Since the tensor product is naturally graded, its derived functors \(\Tor_i\) are graded as well. We will denote the right derived functors of \(\underline{\Hom}\) by \(\underline{\Ext}^i\). Note that since our ring has global dimension \(n\) and graded projective dimension is equal to non-graded projective dimension, we have
\[\Tor_i(-, -) = \underline{\Ext}^i(-, -) = 0\]
whenever \(i > n\).

A persistence module \(M\) is \emph{flat} if the functor \(M \otimes - \colon \PMod \to \PMod\) is exact. It is actually equivalent for a persistence module to be flat as a persistence module and flat as a non-graded \(R\)-module (see e.g.~\cite[Proposition 3.1]{herrmann82}).

A persistence module \(M\) is \emph{projective} (resp.~\emph{injective}), if and only if the functor \(\underline{\Hom}(M, -)\) (resp.~\(\underline{\Hom}(-, M)\)) is exact. Note that it is equivalent for a persistence module to be projective as a persistence module and as a non-graded \(R\)-module. However, this is not the case for injectives. Whenever we call a persistence module injective, we mean injective in the category of persistence modules.

A persistence module \(M\) if \emph{free} if \(M \cong \bigoplus_{i \in \Lambda} R(\mb a_i)\). In fact, projective persistence modules are precisely the free persistence modules (see \cite[Proposition 5]{hoppner81}).

For a persistence module \(M\), the \emph{injective hull of \(M\)} is an injective persistence module \(E\) such that \(M \subseteq E\) and for any submodule \(0 \neq E' \subseteq E\) we have \(M \cap E' \neq 0\). Every persistence module has an injective hull and injective hulls are unique up to isomorphism. We will denote the injective hull of \(M\) by \(E(M)\). A resolution \(0 \to M \xrightarrow{d^{-1}} E^0 \xrightarrow{d^0} E^1 \xrightarrow{d^1} \cdots\) is the \emph{minimal injective resolution of \(M\)} if \(\im d^{i-1} \subseteq E^i\) is the injective hull for every \(i \in \NN\), i.e.~the minimal injective resolution is obtained by chaining injective hulls.

Throughout this paper we will mostly use faces of \(\NN^n\) instead of homogeneous primes of \(R\). A \emph{face} of \(\NN^n\) is a submonoid \(\sigma \subseteq \NN^n\) that is of the form \(\sigma = \langle\mb e_{i_1},\dots,\mb e_{i_k}\rangle\), i.e.~the submonoid generated by \(\mb e_{i_1}, \dots, \mb e_{i_k}\). We denote the \emph{face perpendicular to \(\sigma\)}
\[\sigma^\perp := \langle \mb e_i \mid \mb e_i \not\in \sigma \rangle.\]
We will often denote the face \(\{0\}\) by simply \(0\).

For each face \(\sigma\), we define the corresponding prime
\[\mf p_\sigma := \langle x_i \mid \mb e_i \not\in \sigma\rangle,\]
and for each prime \(\mf p \subseteq R\), we define the corresponding face
\[\sigma_\mf p := \langle \mb e_i \mid x_i \not\in \mf p\rangle.\]
These operations give us an order reversing bijection between the faces of \(\NN^n\) and the homogeneous primes of \(R\) since the homogeneous primes of \(R\) are precisely the ideals generated by a subset of the variables \(x_i\) and thus the operations are mutually inverse.

Let \(\sigma\) be a face. \emph{Localization along} \(\sigma\) is the functor
\[(-)_{\sigma} := (-)_{\mf p_\sigma}.\]
Note that this localization is \emph{homogeneous localization}, i.e.~we invert only the homogeneous elements of \(R \setminus \mf p_\sigma\). Dually, \emph{colocalization along} \(\sigma\) is the functor
\[(-)^{\sigma} := (-)^{\mf p_\sigma} = \underline{\Hom}(R_\sigma,-).\]
We will also use the following limit-constructions for localization and colocalization:
\[\varphi \colon \bigoplus_{\mb d \in \ZZ^n} \varinjlim_{\mb q \in \mb d + \sigma} M_\mb q \to M_\sigma \text{ and } \psi \colon M^\sigma \to \bigoplus_{\mb d \in \ZZ^n} \varprojlim_{\mb q \in \mb d - \sigma} M_\mb q.\]
Note that unlike localization, colocalization is not exact. More specifically, colocalization is only left exact.

For an important example, let \(\sigma\) and \(\tau\) be faces and \(F\) a free persistence module. Now,
\[(F_\tau)^\sigma = \begin{cases} F_\tau, & \text{if } \sigma \subseteq \tau; \\ 0, & \text{if } \sigma \subsetneq \tau. \end{cases}\]

For a face \(\sigma\), we denote
\[\kk(\sigma) := R_\sigma/\mf p_\sigma R_\sigma.\]

For every face \(\sigma\) and \(\mb a \in \ZZ^n\), we define the corresponding upset
\[U_{\mb a,\sigma} := \mb a + \NN^n - \sigma\]
and downset
\[D_{\mb a,\sigma} := \mb a - \NN^n + \sigma.\]
In other words, \(U_{\mb a,\sigma} \subseteq \ZZ^n\) is the smallest upset that contains \(\mb a\) and \(U_{\mb a,\sigma} - \sigma = U_{\mb a,\sigma}\). Similarly, \(D_{\mb a,\sigma} \subseteq \ZZ^n\) is the smallest downset that contains \(\mb a\) and \(D_{\mb a,\sigma} + \sigma = D_{\mb a,\sigma}\).

If \(X \subseteq \ZZ^n\) is a convex subset, we denote \(\kk[X] = \bigoplus_{\mb a \in X} \kk\) and call this the \emph{indicator persistence module over \(X\)}. To be more clear, this is the persistence module corresponding to the functor \(\mb a \mapsto \begin{cases} \kk,& \text{ if } \mb a \in X;\\ 0,& \text{otherwise.} \end{cases}\)

For example, with this notation we have
\[R_\sigma(-\mb a) = \kk[D_{\mb a, \sigma}] \text{ and } \kk(\sigma)(-\mb a) = \kk[\mb a + \sigma - \sigma].\]

\subsection{Matlis duality}\label{subsec: Matlis duality}

The injective hull of \(\kk = R/\langle x_1,\dots x_n \rangle\) is \(E(\kk) = \kk[D_{0,0}]\).

\begin{defn}
The \emph{Matlis duality} functor is the contravariant functor
\[M \mapsto M^\vee := \underline{\Hom}(M, E(\kk)).\]
\end{defn}

The Matlis duality functor is isomorphic to the functor
\[M \mapsto \bigoplus_{\mb a \in \ZZ^n} \Hom_\kk(M_{-\mb a},\kk).\]

Obviously the Matlis duality functor is exact. Note that it is also faithful, i.e.~for a morphism \(\varphi\) we have \(\varphi = 0\), if and only if \(\varphi^\vee = 0\). Additive faithful functors also reflect exact sequences so we get that the Matlis duality functor preserves and reflects exact sequences.

\begin{ex}\label{ex: duals of upset and downset modules}
From the isomorphism \(M^\vee \cong \bigoplus_{\mb a \in \ZZ^n} \Hom_\kk(M_{-\mb a},\kk)\) we see that for any convex subset \(X \subseteq \ZZ^n\), \(\kk[X]^\vee \cong \kk[-X]\). Importantly,
\[\kk[U_{\mb a,\sigma}]^\vee \cong \kk[D_{-\mb a,\sigma}]\]
and
\[\kk[D_{\mb a,\sigma}]^\vee \cong \kk[U_{-\mb a,\sigma}]\]
for all \(\mb a \in \ZZ^n\) and all faces \(\sigma\).
\end{ex}

\begin{lemma}\label{lem: pfd iff isomorphic to double Matlis dual}
A persistence module \(M\) is pointwise finite-dimensional, if and only if the natural morphism \(M \to (M^\vee)^\vee\) is an isomorphism.
\end{lemma}
\begin{proof}
For any \(\mb a \in \ZZ^n\), the linear map \(M_\mb a \to (M^\vee)^\vee_\mb a\) is simply the embedding of the vector space \(M_\mb a\) to the double dual \(\Hom_\kk(\Hom_\kk(M_\mb a,\kk),\kk)\). These are isomorphisms, if and only if \(M\) is pointwise finite-dimensional.
\end{proof}

Next, we will show a duality between flat and injective persistence modules given by the Matlis duality functor. Our proof is simply a graded version of \cite[Theorem 3.2.16]{enochs11}. First, we will need the following lemma, which is a graded version of \cite[Theorem 3.2.11]{enochs11}. The proof in the non-graded setting works for persistence modules as well so we omit it.

\begin{lemma}\label{lem: tensor product of hom isom to double hom}
For all persistence modules \(M\), \(N\) and \(L\), we have a natural morphism
\[\tau_{M,N,L} \colon M \otimes \underline{\Hom}(N,L) \to \underline{\Hom}(\underline{\Hom}(M,N),L).\]
If \(M\) is finitely generated and \(L\) is injective, the morphism \(\tau_{M,N,L}\) is an isomorphism.
\end{lemma}

\begin{prop}\label{prop: flat iff dual inj}
A persistence module \(M\) is flat, if and only if \(M^\vee\) is injective, and \(M\) is injective, if and only if \(M^\vee\) is flat.
\end{prop}
\begin{proof}
The first equivalence was proven e.g.~in \cite[Lemma 11.23]{miller05}: \(M\) is flat, if and only if \(- \otimes M\) is exact. Since the Matlis duality functor preserves and reflects exact sequences, \(- \otimes M\) is exact, if and only if
\[(- \otimes M)^\vee = \underline{\Hom}(- \otimes M,E(\kk))\]
is exact. The adjunction between \(\otimes\) and \(\underline{\Hom}\) gives us an isomorphism of functors
\[\underline{\Hom}(- \otimes M,E(\kk)) \cong \underline{\Hom}(-,\underline{\Hom}(M,E(\kk)) = \underline{\Hom}(-,M^\vee).\]
Hence, \((- \otimes M)^\vee\) is exact, if and only if \(M^\vee\) is injective.

For the second equivalence, we will use the following graded version of Baer's criterion: a persistence module \(E\) is injective, if and only if for all \emph{homogeneous} ideals \(I \subseteq R\), the morphism \(\underline{\Hom}(R,E) \to \underline{\Hom}(I,E)\) is an epimorphism.

Assume that \(M\) is injective and let \(I\) be a homogeneous ideal. By the previous lemma, we have a commutative diagram
\[\begin{tikzcd}[ampersand replacement=\&]
0 \rar \& I \otimes M^\vee \rar \dar{\cong} \& R \otimes M^\vee \dar{\cong} \\
0 \rar \& \underline{\Hom}(I,M)^\vee \rar \& \underline{\Hom}(R,M)^\vee
\end{tikzcd}\]
Since \(M\) is injective and the Matlis duality functor is exact, the bottom row is exact. Hence, the top row is exact. By taking the Matlis dual of the top row, we get a commutative diagram
\[\begin{tikzcd}[ampersand replacement=\&]
\underline{\Hom}(R \otimes M^\vee,E(\kk)) \rar \dar{\cong} \& \underline{\Hom}(I \otimes M^\vee,E(\kk)) \rar \dar{\cong} \& 0 \\
\underline{\Hom}(R,(M^\vee)^\vee) \rar \& \underline{\Hom}(I,(M^\vee)^\vee) \rar \& 0
\end{tikzcd}\]
where the isomorphisms are again given by the adjunction between \(\otimes\) and \(\underline{\Hom}\). Since the top row is exact, the bottom row is exact as well. By Baer's criterion, \((M^\vee)^\vee\) is injective, and so \(M^\vee\) is flat by the first equivalence.

Now, assume that \(M^\vee\) is flat. For any homogeneous ideal \(I\), we again have a commutative diagram
\[\begin{tikzcd}[ampersand replacement=\&]
0 \rar \& I \otimes M^\vee \rar \dar{\cong} \& R \otimes M^\vee \dar{\cong} \\
0 \rar \& \underline{\Hom}(I,M)^\vee \rar \& \underline{\Hom}(R,M)^\vee
\end{tikzcd}\]
This time the top row is exact since \(M^\vee\) is flat, so the bottom row is also exact. Since the Matlis duality functor reflects exact sequences,
\[\underline{\Hom}(R,M) \to \underline{\Hom}(I,M) \to 0\]
is exact. By the graded version of Baer's criterion, \(M\) is injective.
\end{proof}

For any face \(\sigma\) and \(\mb a \in \ZZ^n\), the persistence modules \(R_\sigma(-\mb a) = \kk[U_{\mb a, \sigma}]\) are flat as they are localizations of free modules. Hence, the persistence modules \(\kk[D_{\mb a, \sigma}] \cong \kk[U_{-\mb a, \sigma}]^\vee\) are injective. The embeddings \(R/\mf p_{\sigma} = \kk[\sigma] \hookrightarrow \kk[D_{0, \sigma}]\) being clearly essential show that
\[E(R/\mf p_\sigma)(-\mb a) \cong \kk[D_{\mb a,\sigma}].\]
Further, we get the isomorphisms
\[R_\sigma(-\mb a)^\vee \cong E(R/\mf p_\sigma)(\mb a)\]
and
\[E(R/\mf p_\sigma)(-\mb a)^\vee \cong R_\sigma(\mb a).\]

Every injective persistence module \(E\) has a unique decomposition \(\bigoplus E(R/\mf p_\sigma)(-\mb a) = \bigoplus \kk[D_{\mb a,\sigma}]\) (see \cite[Theorem 1.3.3]{goto78b}). Hence, we can use the previous proposition to decompose pointwise finite-dimensional flat persistence modules uniquely as persistence modules of the form \(\bigoplus R_\sigma(-\mb a) = \bigoplus \kk[U_{\mb a,\sigma}]\). In Section \ref{subsec: Flat cotorsion modules} we will prove a more general form of this decomposition that applies to all \emph{cotorsion} flat modules.

\section{Minimal flat resolutions}
\label{sec: Minimal flat resolutions}

\subsection{Flat covers and minimal flat resolutions}\label{subsec: Flat covers and minimal flat resolutions}

Let \(f \colon F \to M\) be a morphism of persistence modules, where \(F\) is flat. We say that \(f\) is a \emph{flat precover}, if for any flat persistence module \(F'\), the induced morphism
\[\underline{\Hom}(F',F) \to \underline{\Hom}(F',M)\]
is an epimorphism. Equivalently, for any morphism \(g \colon F' \to M\) with \(F'\) flat, there exists a morphism \(h \colon F' \to F\) such that \(g = fh\).

We call \(f\) a \emph{flat cover}, if \(f\) is a flat precover and \(fh = f\) for a morphism \(h \colon F \to F\) implies that \(h\) is an isomorphism.

Let \(F_\bullet \colon \cdots \xrightarrow{d_2} F_1 \xrightarrow{d_1} F_0 \xrightarrow{d_0} M \to 0\) be a resolution. We say that \(F_\bullet\) is a \emph{minimal flat resolution}, if \(F_i \to \im d_i\) is a flat cover for all \(i \in \NN\).

It is a deep result of  Bican, El Bashir, and Enochs proved in \cite{bican01} that all modules over any non-graded ring have flat covers. This result was later strengthened in \cite[Theorem 3.5]{rozas01} to cover graded rings as well. In particular, this means that all \(\ZZ^n\)-persistence modules have flat covers, and further minimal flat resolutions.

Note that flat covers and minimal flat resolutions of a persistence module \(M\) are unique up to isomorphism. Thus we will refer to \emph{the} flat cover and \emph{the} minimal flat resolution of \(M\). We will denote the flat cover of \(M\) by \(F(M)\) and the minimal flat resolution of \(M\) by \(F_\bullet(M)\).

\begin{rem}
Let \(M\) and \(N\) be persistence modules. Any morphism \(f \colon M \to N\) induces a lift \(F_\bullet(f) \colon F_\bullet(M) \to F_\bullet(N)\). These lifts are unique up to homotopy. In fact, if \((F_\bullet,d) \to M \to 0\) is a flat resolution and \((F'_\bullet,d') \to N \to 0\) is a flat resolution such that \(F'_i \to \im d'_i\) is a flat precover for each \(i\), then there exists a chain map \(F_\bullet \to F'_\bullet\) lifting \(f\) and this lift is unique up to homotopy. This can be proven with the same argument as the Comparison Theorem for projective resolutions \cite[Theorem 2.2.6]{weibel94}.
\end{rem}

We recall basic facts about flat precovers and covers. We omit some of the proofs and give references to proofs in a non-graded setting that also work in our setting of persistence modules.

\begin{lemma}\label{lem: flat cover is a summand of flat precover}
Let \(M\) be a persistence module and \(g \colon F \to M\) a flat precover. Then, we have an isomorphism \(F \cong F(M) \oplus F'\) such that \(g|_{F(M)}\) is the flat cover and \(F' \subseteq \ker g\).
\end{lemma}
\begin{proof}
See \cite[Theorem 1.2.7]{xu96}, i.e.~dualize the proof of \cite[Proposition 1.2.2]{xu96}.
\end{proof}

\begin{cor}\label{cor: flat precover is cover iff no nonzero summand in kernel}
A flat precover \(f \colon F \to M\) is the flat cover, if and only if the only direct summand \(F' \subseteq F\) such that \(F' \subseteq \ker f\), is \(F' = 0\).
\end{cor}

\begin{lemma}
Let \(M\) be a persistence module and \(\cdots \xrightarrow{d_2} F_1 \xrightarrow{d_1} F_0 \xrightarrow{d_0} M \to 0\) the minimal flat resolution. Let \(\cdots \xrightarrow{d'_2} F'_1 \xrightarrow{d'_1} F'_0 \xrightarrow{d'_0} M \to 0\) be a resolution such that \(F'_i \to \im d'_i\) is a flat precover for all \(i \in \NN\). Then, there exist chain maps \(\varphi \colon F_\bullet \to F'_\bullet\) and \(\psi \colon F'_\bullet \to F_\bullet\) that lift \(\id_M\) such that \(\psi \varphi = \id_{F_\bullet}\). Hence, \(F_\bullet\) is a direct summand of \(F'_\bullet\).
\end{lemma}
\begin{proof}
We denote \(K_i = \im d_{i+1}\), and \(K'_i = \im d'_{i+1}\) for all \(i \geq -1\). Fix \(i \in \NN\) and assume that we have morphisms \(K_{i-1} \to K'_{i-1}\) and \(K'_{i-1} \to K_{i-1}\) such that the composition \(K_{i-1} \to K'_{i-1} \to K_{i-1}\) is the identity. Since \(F_i \to K_{i-1}\) and \(F'_i \to K'_{i-1}\) are flat precovers, we get morphisms between \(F_i\) and \(F'_i\) making the diagrams
\[\begin{tikzcd}[column sep=small]
0 \rar & K_i \rar & F_i \dar \rar & K_{i-1} \dar \rar & 0 \\
0 \rar & K'_i \rar & F'_i \rar & K'_{i-1} \rar & 0
\end{tikzcd}
\quad
\begin{tikzcd}[column sep=small]
0 \rar & K_i \rar & F_i \rar & K_{i-1} \rar & 0 \\
0 \rar & K'_i \rar & F'_i \uar \rar & K'_{i-1} \uar \rar & 0
\end{tikzcd}\]
commute. Since \(F_i \to K_{i-1}\) is the flat cover and \(K_{i-1} \to K'_{i-1} \to K_{i-1}\) is the identity, these diagrams show that the composition \(F_i \to F'_i \to F_i\) has to be an isomorphism. We can assume that it is the identity. The rows in the diagrams are exact, so we get induced morphisms \(K_i \to K'_i\) and \(K'_i \to K_i\). The composition \(K_i \to K'_i \to K_i\) is the identity since \(F_i \to F'_i \to F_i\) is the identity.

Starting from \(i = 0\) and the identity morphism \(M \to M\), we can build the chain maps \(\varphi\) and \(\psi\) inductively using this process.
\end{proof}

\begin{lemma}
For morphisms \(f_i \colon F_i \to M_i\), \(i \in \Lambda\), the direct product
\[\prod_{i \in \Lambda} f_i \colon \prod_{i \in \Lambda} F_i \to \prod_{i \in \Lambda} M_i\]
is a flat precover, if and only if \(f_i\) is a flat precover for all \(i \in \Lambda\).
\end{lemma}
\begin{proof}
Follows from the natural isomorphisms \(\underline{\Hom}(F',\prod_{i \in \Lambda} F_i) \cong \prod_{i \in \Lambda} \underline{\Hom}(F',F_i)\) and \(\underline{\Hom}(F',\prod_{i \in \Lambda} M_i) \cong \prod_{i \in \Lambda} \underline{\Hom}(F',M_i)\).
\end{proof}

\begin{lemma}\label{lem: direct sum of flat covers}
Let \(f_1 \colon F_1 \to M_1\) and \(f_2 \colon F_2 \to M_2\) be morphisms. The direct sum \(f_1 \oplus f_2 \colon F_1 \oplus F_2 \to M_1 \oplus M_2\) is the flat cover, if and only if \(f_1\) and \(f_2\) are flat covers.
\end{lemma}
\begin{proof}
See \cite[Theorem 1.2.10]{xu96}, i.e.~dualize the proof of \cite[Theorem 1.2.5]{xu96}.
\end{proof}

\subsection{Cotorsion}\label{subsec: Cotorsion}

Every flat precover is clearly an epimorphism. On the other hand, an epimorphism from a flat persistence module can fail to be a precover. To see one reason why, let \(F\) be a flat persistence module and take a short exact sequence
\[0 \to K \to F \to M \to 0\]
of persistence modules. For all flat persistence modules \(F'\) we have an exact sequence
\[\underline{\Hom}(F',F) \to \underline{\Hom}(F',M) \to \underline{\Ext}^1(F',K).\]
This shows that the epimorphism \(F \to M\) is a flat precover, if and only if the morphism \(\underline{\Hom}(F',M) \to \underline{\Ext}^1(F',K)\) is \(0\) for all flat persistence modules \(F'\). An easy case of when this holds is when \(K\) is a \emph{cotorsion} persistence module:

\begin{defn}
A persistence module \(M\) is \emph{cotorsion}, if for all flat persistence modules \(F\) we have \(\underline{\Ext}^1(F,M)=0\).
\end{defn}

\begin{lemma}\label{lem: epimorphism flat precover when ker cotorsion}
Let \(M\) be a persistence module and \(F\) a flat persistence module. Then, any epimorphism \(f \colon F \to M\) such that \(\ker f\) is cotorsion is a flat precover.
\end{lemma}

For an epimorphism \(f \colon F \to M\) where \(F\) is flat, the condition that \(\ker f\) is cotorsion is stronger than \(f\) being a flat precover. For flat covers however, the kernel is always cotorsion.

\begin{prop}\label{prop: kernel of flat cover is cotorsion}
Let \(M\) be a persistence module and \(f \colon F(M) \to M\) the flat cover of \(M\). Then, \(\ker f\) is cotorsion.
\end{prop}
\begin{proof}
We adapt the proof of \cite[Lemma 2.2]{enochs84} to our graded setting as follows. Let \(F\) be a flat persistence module. Let \(N\) be a projective persistence module with a submodule \(S \subseteq N\) such that \(N/S = F\). Now, we get an exact sequence
\[\underline{\Hom}(N,\ker f) \to \underline{\Hom}(S,\ker f) \to \underline{\Ext}^1(F,\ker f) \to \underline{\Ext}^1(N,\ker f) = 0.\]
The argument in \cite[Lemma 2.2]{enochs84} shows that \(\underline{\Hom}(N,\ker f)_0 \to \underline{\Hom}(S,\ker f)_0\) is surjective. Thus \(\underline{\Ext}^1(F,\ker f)_0 = 0\).

Since any shift of a flat persistence module is still flat, this implies that \(\underline{\Ext}^1(F,\ker f)_\mb a = \underline{\Ext}^1(F(-\mb a),\ker f)_0 = 0\) for all flat persistence modules \(F\) and \(\mb a \in \ZZ^n\). Thus \(\ker f\) is cotorsion.
\end{proof}

\begin{cor}\label{cor: flat cover is cotorsion}
A persistence module \(M\) is cotorsion, if and only if the flat cover \(F(M)\) is cotorsion.
\end{cor}
\begin{proof}
Let \(F\) be a flat persistence modules. Let \(K\) be the kernel of the flat cover \(F(M) \to M\). It is easy to see that \(\underline{\Ext}^2(F,K) = 0\). The exact sequence
\[0 = \underline{\Ext}^1(F,K) \to \underline{\Ext}^1(F,F(M)) \to \underline{\Ext}^1(F,M) \to \underline{\Ext}^2(F,K) = 0\]
then shows that \(\underline{\Ext}^1(F,F(M)) \cong \underline{\Ext}^1(F,M)\). Thus, \(M\) is cotorsion, if and only if \(F(M)\) is cotorsion.
\end{proof}

\begin{lemma}\label{lem: Matlis dual cotorsion}
For a persistence module \(M\), and an injective persistence module \(E\), the persistence module \(\underline{\Hom}(M,E)\) is cotorsion.
\end{lemma}
\begin{proof}
The non-graded proof of \cite[Lemma 2.1]{enochs84} works in our setting as well, but requires the reader to be familiar with pure injective modules. For convenience, we give a short proof without explicitly using pure injectives.

Let \(F\) be a flat persistence module and \(0 \to K \to P \to F \to 0\) a short exact sequence, where \(P\) is projective. From this, we get an exact sequence
\[\underline{\Hom}(P,\underline{\Hom}(M,E)) \to \underline{\Hom}(K,\underline{\Hom}(M,E)) \to \underline{\Ext}^1(F,\underline{\Hom}(M,E)) \to 0.\]
By the tensor-hom adjunction, this sequence is isomorphic to
\[\underline{\Hom}(P \otimes M,E) \to \underline{\Hom}(K \otimes M,E) \to \underline{\Ext}^1(F,\underline{\Hom}(M,E)) \to 0.\]
The exact sequence \(0 = \Tor_1(F,M) \to K \otimes M \to P \otimes M\) shows that \(K \otimes M \to P \otimes M\) is a monomorphism. Hence, the morphism \(\underline{\Hom}(P \otimes M,E) \to \underline{\Hom}(K \otimes M,E)\) is an epimorphism, so \(\underline{\Ext}^1(F,\underline{\Hom}(M,E)) = 0\).
\end{proof}

\begin{ex}
Persistent homology can fail to be cotorsion. However, persistent cohomology is always cotorsion, as it is the Matlis dual of persistent homology.
\end{ex}

\begin{cor}\label{cor: pfd cotorsion}
Pointwise finite-dimensional persistence modules are cotorsion.
\end{cor}

\begin{rem}\label{rem: flat cover and inj hull defns agree with Miller's defn}
Let \(M\) be a persistence module with an epimorphism \(f \colon \bigoplus_{i=1}^m R_{\sigma_i}(-\mb a_i) \to M\). Since \(\ker f\) is pointwise finite-dimensional, it is cotorsion and so \(f\) is a flat precover. By Lemma \ref{lem: flat cover is a summand of flat precover}, there exists a persistence module \(F'\) such that \(F(M) \oplus F' \cong \bigoplus_{i=1}^m R_{\sigma_i}(-\mb a_i)\). Since \(F(M)^\vee\) is injective and \(\dim_\kk (F(M)^\vee)_\mb a \leq m\) for all \(\mb a \in \ZZ^n\), we have a finite decomposition \(F(M)^\vee \cong \bigoplus_{i=1}^k E(R/\mf p_{\tau_i})(-\mb b_i)\) and further \(F(M) \cong \bigoplus_{i=1}^k R_{\tau_i}(\mb b_i)\). Similarly we see that \(F' \cong \bigoplus_{i=1}^l R_{\rho_i}(\mb c_i)\) with \(m = k+l\).

In particular, this shows that our definition of the flat cover of \(M\) is equivalent to Miller's definition of a minimal flat cover of \(M\) in \cite[Definition 5.8]{miller20b}, i.e.~to an epimorphism \(\bigoplus_{i=1}^m R_{\sigma_i}(-\mb a_i) \to M\) where \(m\) is the smallest number for which such an epimorphism exists. Dually, if there exists a monomorphism \(M \to \bigoplus_{i=1}^l E(R/\mf p_{\sigma_i})(-\mb a_i)\), then our definition of injective hull is equivalent to Miller's definition of minimal injective hull \cite[Definition 5.6]{miller20b}, i.e.~to a monomorphism \(M \to \bigoplus_{i=1}^l E(R/\mf p_{\sigma_i})(-\mb a_i)\) where \(l\) is the smallest number for which such a monomorphism exists.
\end{rem}

\begin{lemma}\label{lem: product of cotorsion is cotorsion}
For a family of persistence modules \((M_i)_{i \in \Lambda}\), the product \(\prod_{i \in \Lambda} M_i\) is cotorsion, if and only if each \(M_i\) is cotorsion.
\end{lemma}
\begin{proof}
The claim follows from the identity \(\underline{\Ext}^1(F, \prod_{i \in \Lambda} M_i) \cong \prod_{i \in \Lambda} \underline{\Ext}^1(F,M_i)\).
\end{proof}

\begin{ex}\label{ex: flat cover of Tp}
With the help of cotorsion, we can give our first non-finitely generated example of flat covers. We choose a face \(\sigma\), and let
\[M = \bigoplus_{\mb a \in \ZZ\sigma^\perp} \kk(\sigma)(-\mb a)^{(\beta_\mb a)}\]
for some sets \(\beta_\mb a\) for each \(\mb a \in \ZZ\sigma^\perp\). Note that these sets are uniquely determined by \(M\). Let \(F = \prod_{\mb a \in \ZZ\sigma^\perp} R_\sigma(-\mb a)^{(\beta_\mb a)}\) and let \(p \colon F \to F/\mf p_\sigma F = \prod_{\mb a \in \ZZ\sigma^\perp} \kk(\sigma)(-\mb a)^{(\beta_\mb a)} = M\) be the projection. We will show that \(p\) is the flat cover of \(M\).

We start by showing that \(p\) is a flat precover. Note that \(\mf p_\sigma R_\sigma(-\mb a)\) is cotorsion since it is pointwise finite-dimensional. Hence, any product of persistence modules of the form \(\mf p_\sigma R_\sigma(-\mb a)\) is cotorsion by Lemma \ref{lem: product of cotorsion is cotorsion}. Since
\[\mf p_\sigma F = \prod_{\mb a \in \ZZ\sigma^\perp} (\mf p_\sigma R_\sigma(-\mb a))^{(\beta_\mb a)} \subseteq \prod_{\mb a \in \ZZ\sigma^\perp} (\mf p_\sigma R_\sigma(-\mb a))^{\beta_\mb a}\]
is the embedding of a direct summand, \(\mf p_\sigma F\) is cotorsion. Since \(\ker p = \mf p_\sigma F\) is cotorsion, Lemma \ref{lem: epimorphism flat precover when ker cotorsion} shows that \(p\) is a flat precover.

By Lemma \ref{lem: flat cover is a summand of flat precover}, we can then write \(F = F(M) \oplus A\) with \(p(A) = 0\), i.e.~\(A \subseteq \mf p_\sigma F\). Since \(A \subseteq \mf p_\sigma F(M) \oplus \mf p_\sigma A \subseteq F(M) \oplus A\), we must have \(A \subseteq \mf p_\sigma A\). Further, we get
\[A = \bigcap_{i \in \ZZ_+} \mf p_\sigma^i A \subseteq \bigcap_{i \in \ZZ_+} \mf p_\sigma^i F = 0,\]
so \(F = F(M)\) and \(p\) is the flat cover of \(M\).
\end{ex}

\begin{ex}\label{ex: q-tame is cotorsion}
Let \(M\) be a persistence module and set \(\mf m = \langle x_1,\dots,x_n \rangle\). The persistence module \(M/\mf m M\) is simply a \(\ZZ^n\)-graded \(\kk\)-vector space, so
\[M/\mf m M \cong \bigoplus_{\mb a \in \ZZ^n} \kk(-\mb a)^{(\beta_{\mb a_i})} = \prod_{\mb a \in \ZZ^n} \kk(-\mb a)^{(\beta_{\mb a_i})}\]
This is cotorsion as it is a direct summand of \(\prod_{\mb a \in \ZZ^n} \kk(-\mb a)^{\beta_{\mb a_i}}\), which is cotorsion as a product of cotorsion persistence modules. If \(\mf m M\) is also cotorsion, the short exact sequence
\[0 \to \mf m M \to M \to M/\mf m M \to 0\]
shows that \(M\) is cotorsion.

Now, let \(\mb P\) be a preordered set with an order preserving injection \(i \colon \ZZ^n \to \mb P\). Consider a functor \(N \colon \mb P \to \Vect\), such that \(N(a < b)\) is of finite rank for all \(a < b \in \mb P\). Let \(M\) be the \(\ZZ^n\)-persistence module corresponding to the functor \(N \circ i\). Now \(\mf m M\) is cotorsion as it is pointwise finite-dimensional. Thus \(M\) is cotorsion. In particular, this shows that discretizations of q-tame persistence modules are cotorsion.
\end{ex}

\begin{lemma}\label{lem: cotorsion iff cotorsion resolution}
Let \(M\) be a persistence module with a resolution
\[\cdots \xrightarrow{d_2} C_1 \xrightarrow{d_1} C_0 \xrightarrow{d_0} M \to 0\]
such that \(C_i\) is cotorsion for each \(i \in \NN\). Then, \(M\) is cotorsion.
\end{lemma}
\begin{proof}
We adapt the non-graded proof of \cite[Lemma 2.5]{marley16} which we can simplify as our ring has global dimension \(n\). For all flat persistence modules \(F\) and \(i \in \NN\) we have exact sequences
\[0 = \underline{\Ext}^{i+1}(F,C_i) \to \underline{\Ext}^{i+1}(F,\im d_i) \to \underline{\Ext}^{i+2}(F,\im d_{i+1}) \to \underline{\Ext}^{i+2}(F,C_i) = 0.\]
Thus \(\underline{\Ext}^1(F,M) \cong \underline{\Ext}^2(F,\im d_1) \cong \cdots \cong \underline{\Ext}^{n+1}(F,\im d_n) = 0\) so \(M\) is cotorsion.
\end{proof}

\begin{lemma}\label{lem: coloc preserves cotorsion}
Let \(M\) be a cotorsion persistence module, and \(\sigma\) a face of \(\NN^n\). Then \(M^\sigma\) is a cotorsion persistence module.
\end{lemma}
\begin{proof}
Let \(F\) be a flat persistence module with a short exact sequence \(0 \to K \to P \to F \to 0\) where \(P\) is projective. From the tensor-hom adjunction we get the commutative diagram
\[\begin{tikzcd}[column sep=small]
\underline{\Hom}(P,M^\sigma) \dar{\cong}\rar & \underline{\Hom}(K,M^\sigma) \dar{\cong}\rar & \underline{\Ext}^1(F,M^\sigma) \rar & \underline{\Ext}^1(P,M^\sigma) = 0 \\
\underline{\Hom}(P_\sigma,M) \rar & \underline{\Hom}(K_\sigma,M) \rar & \underline{\Ext}^1(F_\sigma,M) = 0
\end{tikzcd}\]
where the rows are exact. Thus the morphism \(\underline{\Hom}(P,M^\sigma) \to \underline{\Hom}(K,M^\sigma)\) is an epimorphism and so \(\underline{\Ext}^1(F,M^\sigma) = 0\).
\end{proof}

\begin{lemma}\label{lem: coloc preserves short exact sequence if ker cotorsion}
Let \(\sigma\) be a face and \(0 \to K \to M \to N \to 0\) a short exact sequence of persistence modules. If \(K\) is cotorsion, then \(0 \to K^\sigma \to M^\sigma \to N^\sigma \to 0\) is exact. Hence, colocalization along \(\sigma\) preserves exact sequences with cotorsion kernels.
\end{lemma}
\begin{proof}
The result follows from the exact sequence
\[0 \to K^\sigma \to M^\sigma \to N^\sigma \to \underline{\Ext}^1(R_\sigma,K) = 0.\]
\end{proof}

\subsection{Flat cotorsion modules}\label{subsec: Flat cotorsion modules}

In this section, we will show that flat cotorsion persistence modules have unique decompositions that are similar to the decompositions of injective persistence modules. This kind of decomposition was first proven by Enochs in \cite[Theorem]{enochs84} for non-graded flat cotorsion modules over a Noetherian ring \(A\). The theorem of Enochs says that an \(A\)-module \(F\) is flat and cotorsion, if and only if \(F \cong \prod_{\mf p \in \spec A} T_\mf p\), where each \(T_\mf p\) is the \(\mf p\)-adic completion of a free \(A_\mf p\)-module.

Our proof is heavily inspired by the proof of Enochs' decomposition theorem given by Xu in \cite[Section 4.1]{xu96} with two major simplifications. First, we have only finitely many faces instead of possibly infinitely many primes. Second, we can avoid discussing completions entirely, since these completions have a very explicit form in the case of \(\ZZ^n\)-graded \(\kk[x_1,\dots,x_n]\)-modules.

\begin{lemma}\label{lem: flat cotorsion iff double dual splits}
Let \(F\) be a flat persistence module. Then \(F\) is cotorsion, if and only if \(F \to (F^\vee)^\vee\) splits.
\end{lemma}
\begin{proof}
Direct summands of cotorsion persistence modules are cotorsion by Lemma \ref{lem: product of cotorsion is cotorsion}. Since \((F^\vee)^\vee\) is cotorsion by Lemma \ref{lem: Matlis dual cotorsion}, we see that \(F\) is cotorsion if \(F \to (F^\vee)^\vee\) splits.

Assume then that \(F\) is cotorsion. Let \(M\) be any persistence module. We have a morphism \((F^\vee)^\vee \otimes M \to ((F \otimes M)^\vee)^\vee\) given by \(f \otimes m \mapsto (\varphi \mapsto f(\varphi(- \otimes m)))\). With this morphism, the diagram
\[\begin{tikzcd}
F \otimes M \rar\ar{dr} & (F^\vee)^\vee \otimes M \dar \\
& ((F \otimes M)^\vee)^\vee
\end{tikzcd}\]
commutes. Since the morphism \(F \otimes M \to ((F \otimes M)^\vee)^\vee\) is a monomorphism, the morphism \(F \otimes M \to (F^\vee)^\vee \otimes M\) also is a monomorphism. Hence, the exact sequence
\[0 = \Tor_1((F^\vee)^\vee,M) \to \Tor_1((F^\vee)^\vee/F,M) \to F \otimes M \to (F^\vee)^\vee \otimes M\]
shows that \(\Tor_1((F^\vee)^\vee/F,M) = 0\), i.e.~\((F^\vee)^\vee/F\) is flat. Now, the exact sequence
\[\underline{\Hom}((F^\vee)^\vee,F) \to \underline{\Hom}(F,F) \to \underline{\Ext}^1((F^\vee)^\vee/F,F) = 0\]
shows that \(\underline{\Hom}((F^\vee)^\vee,F) \to \underline{\Hom}(F,F)\) is an epimorphism. Hence, \(\id \colon F \to F\) factors through \(F \to (F^\vee)^\vee\), so \(F \to (F^\vee)^\vee\) splits.
\end{proof}

\begin{lemma}
Let \(\sigma\) be a face and \(F \cong \prod_{\mb a \in \ZZ\sigma^\perp} R_\sigma(-\mb a)^{(\beta_\mb a)}\) for some sets \(\beta_\mb a\). If \(F \cong A \oplus B\), then \(A \cong \prod_{\mb a \in \ZZ\sigma^\perp} R_\sigma(-\mb a)^{(\rho_\mb a)}\) for some sets \(\rho_\mb a\).
\end{lemma}
\begin{proof}
In Example \ref{ex: flat cover of Tp} we saw that \(F \to F/\mf p_\sigma F\) is a flat cover. By Lemma \ref{lem: direct sum of flat covers} \(A \to A/\mf p_\sigma A\) is a flat cover as well. Now, since \(A/\mf p_\sigma A\) is a \(\ZZ^n\)-graded \(\kk(\sigma)\)-module, it is free, i.e.~\(A/\mf p_\sigma A \cong \bigoplus_{i \in \Lambda} \kk(\sigma)(-\mb a_i)\); see e.g.~\cite[Theorem 1.1.4]{goto78b}. As flat covers are unique, Example \ref{ex: flat cover of Tp} again shows us that \(A \cong \prod_{\mb a \in \ZZ\sigma^\perp} R_\sigma(-\mb a)^{(\rho_\mb a)}\) for some sets \(\rho_\mb a\).
\end{proof}

\begin{lemma}\label{lem: flat cotorsion decomposition direct summand}
Let \(F \cong \prod_{\sigma,\mb a \in \ZZ\sigma^\perp} R_\sigma(-\mb a)^{(\beta^\sigma_\mb a)}\) for some sets \(\beta^\sigma_\mb a\). If \(F \cong A \oplus B\), then \(A \cong \prod_{\sigma,\mb a \in \ZZ\sigma^\perp} R_\sigma(-\mb a)^{(\rho^\sigma_\mb a)}\) for some sets \(\rho^\sigma_\mb a\).
\end{lemma}
\begin{proof}
For each face \(\sigma = \langle \mb e_{i_1},\dots, \mb e_{i_k} \rangle\) with \(i_1 < \cdots < i_k\), we set \(\abs \sigma := k\). For each \(k = 0,\dots,n\) we set
\[F_{\geq k} := \prod_{\abs \sigma \geq k,\mb a \in \ZZ\sigma^\perp} R_\sigma(-\mb a)^{(\beta^\sigma_\mb a)} \text{ and } F_k := F_{\geq k}/F_{\geq k+1} = \prod_{\abs \sigma = k,\mb a \in \ZZ\sigma^\perp} R_\sigma(-\mb a)^{(\beta^\sigma_\mb a)}.\]
Since \(\underline{\Hom}(R_\sigma,R_\tau) = 0\) for faces \(\sigma \not\subseteq \tau\), we see that any morphism \(F_{\geq k} \to F\) has to have its image inside \(F_{\geq k}\). Hence, for each \(k\), the composition \(i_k \colon F_{\geq k} \to F \to A \to F \to F_{\geq k}\) satisfies \(i_k \circ i_k = i_k\). Thus \(F_{\geq k} = \im i_k \oplus \ker i_k\). For each \(k\), we set \(A_{\geq k} := \im i_k\), \(A_k := A_{\geq k}/A_{\geq k+1}\) if \(k < n\), and \(A_n := A_{\geq n}\).

The short exact sequences
\[0 \to F_{\geq k+1} \to F_{\geq k} \to F_k \to 0\]
split so their direct summands
\[0 \to A_{\geq k+1} \to A_{\geq k} \to A_k \to 0\]
also split. Thus \(A = A_{\geq 0} \cong \bigoplus_k A_k\).

Again, since \(\underline{\Hom}(R_\sigma,R_\tau) = 0\) for faces \(\sigma \not\subseteq \tau\), we see that each composition \(\varphi_k \colon F_k \to A_k \to F_k\) is a direct sum of morphisms \(\varphi_\sigma \colon \prod_{\mb a \in \ZZ\sigma^\perp} R_\sigma(-\mb a)^{(\beta^\sigma_\mb a)} \to \prod_{\mb a \in \ZZ\sigma^\perp} R_\sigma(-\mb a)^{(\beta^\sigma_\mb a)}\). Thus \(A_k \cong \bigoplus_{\abs \sigma = k} \im \varphi_\sigma\). Since \(\varphi_k \circ \varphi_k = \varphi_k\), we see that \(\varphi_\sigma \circ \varphi_\sigma = \varphi_\sigma\) for each face \(\sigma\). By the previous lemma, \(\im \varphi_\sigma \cong \prod_{\mb a \in \ZZ\sigma^\perp} R_\sigma(-\mb a)^{(\rho^\sigma_\mb a)}\) for some sets \(\rho^\sigma_\mb a\). Finally,
\[A \cong \bigoplus_\sigma \im \varphi_\sigma = \prod_\sigma \im \varphi_\sigma = \prod_{\sigma,\mb a \in \ZZ\sigma^\perp} R_\sigma(-\mb a)^{(\rho^\sigma_\mb a)}.\]
\end{proof}

\begin{thm}\label{thm: flat cotorsion structure}
A persistence module \(F\) is flat and cotorsion, if and only if
\[F \cong \prod_{\sigma,\mb a \in \ZZ\sigma^\perp} R_\sigma(-\mb a)^{(\beta^\sigma_\mb a)}\]
for some sets \(\beta^\sigma_\mb a\). The sets \(\beta^\sigma_\mb a\) are uniquely determined by \(F\).
\end{thm}
\begin{proof}
First, assume that \(F\) is flat and cotorsion. Since \(F^\vee\) is injective, we have a decomposition \(F^\vee \cong \bigoplus_{\sigma,\mb a \in \ZZ\sigma^\perp} E(R/\mf p_\sigma)(\mb a)^{(\tau_\mb a^\sigma)}\), and thus \((F^\vee)^\vee \cong \prod_{\sigma,\mb a \in \ZZ\sigma^\perp} R_\sigma(-\mb a)^{\tau_\mb a^\sigma}\). For every face \(\sigma\) and \(\mb a \in \ZZ\sigma^\perp\), there exists a set \(\rho^\sigma_\mb a\) such that \(R_\sigma(-\mb a)^{\tau_\mb a^\sigma} \cong R_\sigma(-\mb a)^{(\rho_\mb a^\sigma)}\). By Lemma \ref{lem: flat cotorsion iff double dual splits}, \(F\) is a direct summand of \((F^\vee)^\vee \cong \prod_{\sigma,\mb a \in \ZZ\sigma^\perp} R_\sigma(-\mb a)^{(\rho_\mb a^\sigma)}\). Hence, Lemma \ref{lem: flat cotorsion decomposition direct summand} shows that \(F \cong \prod_{\sigma,\mb a \in \ZZ\sigma^\perp} R_\sigma(-\mb a)^{(\beta^\sigma_\mb a)}\) for some sets \(\beta^\sigma_\mb a\).

Assume then that \(F \cong \prod_{\sigma,\mb a \in \ZZ\sigma^\perp} R_\sigma(-\mb a)^{(\beta^\sigma_\mb a)}\). Since \(R_\sigma(-\mb a)^{(\beta^\sigma_\mb a)}\) is a direct summand of \(R_\sigma(-\mb a)^{\beta^\sigma_\mb a}\), \(F\) is a direct summand of \(\prod_{\sigma,\mb a \in \ZZ\sigma^\perp} R_\sigma(-\mb a)^{\beta^\sigma_\mb a} \cong \big(\bigoplus_{\sigma,\mb a \in \ZZ\sigma^\perp} E(R/\mf p_\sigma)(\mb a)^{(\beta^\sigma_\mb a)}\big)^\vee\), which is flat and cotorsion. Thus \(F\) is flat and cotorsion.

We still need to show that the sets \(\beta^\sigma_\mb a\) are uniquely determined by \(F\). Let \(\tau\) be a face and \(\mb b \in \ZZ\tau^\perp\). Note that
\[F^\tau \cong \big(\prod_{\sigma,\mb a \in \ZZ\sigma^\perp} R_\sigma(-\mb a)^{(\beta^\sigma_\mb a)}\big)^\tau \cong \prod_{\sigma \supseteq \tau,\mb a \in \ZZ\sigma^\perp} R_\sigma(-\mb a)^{(\beta^\sigma_\mb a)}.\]
Now
\[\kk(\tau) \otimes F^\tau \cong F^\tau/\mf p_\tau F^\tau \cong \prod_{\mb a \in \ZZ\tau^\perp} \kk(\tau)(-\mb a)^{(\beta^\tau_\mb a)}.\]
Finally, the cardinality of the set \(\beta^\tau_\mb b\) is equal to \(\dim_\kk (\kk(\tau) \otimes F^\tau)_\mb b\). Thus the sets are uniquely determined by \(F\).
\end{proof}

\begin{rem}
If \(M\) is a cotorsion persistence module, the minimal flat resolution \(F_\bullet(M) \to M \to 0\) consists of flat cotorsion persistence modules by Corollary \ref{cor: flat cover is cotorsion} and Proposition \ref{prop: kernel of flat cover is cotorsion}. Thus, for every \(i\) we have a decomposition \(F_i(M) \cong \prod_{\sigma,\mb a \in \ZZ\sigma^\perp} R_\sigma(-\mb a)^{(\beta^\sigma_{i,\mb a})}\) where the sets \(\beta^\sigma_{i,\mb a}\) are uniquely determined by \(M\).
\end{rem}

\begin{rem}\label{rem: flat precover is cover iff no R_sigma summand inside kernel}
Let \(f \colon F \to M\) be a flat precover such that \(F\) is cotorsion. Now, using the above decomposition theorem, we can rewrite Corollary \ref{cor: flat precover is cover iff no nonzero summand in kernel} as follows: \(f\) is the flat cover, if and only if there is no direct summand \(R_\sigma(-\mb a) \subseteq F\) such that \(R_\sigma(-\mb a) \subseteq \ker f\).
\end{rem}

In the proof of the previous theorem, we proved the following.

\begin{cor}\label{cor: coloc preserves flat cotorsion}
Let \(F\) be a flat cotorsion persistence module and \(\sigma\) a face. Then \(F^\sigma\) is flat and cotorsion.
\end{cor}

\begin{lemma}\label{lem: loc and coloc preserve hulls and covers}
Let \(M\) be a cotorsion persistence module, \(f \colon F(M) \to M\) the flat cover of \(M\), and \(\sigma\) a face. Then, the morphism \(f^\sigma \colon F(M)^\sigma \to M^\sigma\) is the flat cover.
\end{lemma}
\begin{proof}
Since \(\ker f\) is cotorsion by Proposition \ref{prop: kernel of flat cover is cotorsion}, the morphism \(f^\sigma\) is an epimorphism by Lemma \ref{lem: coloc preserves short exact sequence if ker cotorsion}. Also, \(F(M)^\sigma\) is flat by Corollary \ref{cor: coloc preserves flat cotorsion} and \(\ker f^\sigma = (\ker f)^\sigma\) is cotorsion by Lemma \ref{lem: coloc preserves cotorsion}. Thus \(f^\sigma\) is a flat precover by Lemma \ref{lem: epimorphism flat precover when ker cotorsion}.

If \(f^\sigma\) is not a flat cover, then by Corollary \ref{cor: flat precover is cover iff no nonzero summand in kernel} there exists a direct summand \(0 \neq F \subseteq F(M)^\sigma\) such that \(f^\sigma(F) = 0\). Consider the following commutative diagram
\[\begin{tikzcd}
F(M)^\sigma \rar\dar & M^\sigma \dar \\
F(M) \rar & M
\end{tikzcd}\]
Since \(F(M)\) is flat and cotorsion, the morphism \(F(M)^\sigma \to F(M)\) is an embedding of a direct summand. Thus it also embeds \(F\) as a direct summand of \(F(M)\) such that \(f(F) = 0\). This is a contradiction as \(f\) is a flat cover. Hence \(f^\sigma\) is a flat cover.
\end{proof}

\begin{cor}\label{cor: loc and coloc preserve min resolutions}
Let \(M\) be a cotorsion persistence module and \(F_\bullet(M) \to M \to 0\) the minimal flat resolution. Then \(F_\bullet(M)^\sigma \to M^\sigma \to 0\) is the minimal flat resolution.
\end{cor}

\section{Duality between injective hulls and flat covers}
\label{sec: Duality between injcetive hulls and flat covers}

\subsection{Generator and cogenerator functors}\label{subsec: Generator and cogenerator functors}

\begin{defn}
For all faces \(\sigma\), we define the \emph{generator functor along} \(\sigma\),
\[\Topp_\sigma \colon \ZZ^n\text{-}\grRMod \to \ZZ^n\text{-}\grRMod, \Topp_\sigma M = \kk(\sigma) \otimes M^\sigma,\]
and the \emph{cogenerator functor along} \(\sigma\),
\[\Soc_\sigma \colon \ZZ^n\text{-}\grRMod \to \ZZ^n\text{-}\grRMod, \Soc_\sigma M = \underline{\Hom}(\kk(\sigma),M_\sigma).\]
\end{defn}

\begin{rem}
These functors are named after the related functors of Miller in \cite{miller20}: the closed generator functor \cite[Definition 11.17]{miller20} and the global closed cogenerator functor \cite[Definition 4.15.1]{miller20}.
\end{rem}

Let \(M\) be a persistence module, \(E\) an injective persistence module, and \(i \colon M \to E\) a monomorphism. It is known that \(i\) is an injective hull, if and only if \(\Soc_\sigma i\) is an isomorphism for all faces \(\sigma\) (a reader unfamiliar with this can see this by dualizing the proof of the next proposition). The functors \(\Topp_\sigma\) dually detect flat covers and this has been proven in the non-graded setting. The left-to-right implication was first proven by Enochs and Xu in the proof of \cite[Theorem 2.2]{enochs97}, and the right-to-left implication was proven by Dailey in his thesis \cite[Proposition 4.2.7]{dailey16}. For completeness, we give a version of this proof adapted to persistence modules.

\begin{prop}[Graded version of {\cite[Proposition 4.2.7]{dailey16}}]\label{prop: flat cover iff top isomorphisms}
Let \(M\) be a cotorsion persistence module, \(F\) a flat persistence module, and \(f \colon F \to M\) an epimorphism such that \(\ker f\) is cotorsion. Then, \(f\) is the flat cover, if and only if \(\Topp_\sigma f\) is an isomorphism for all faces \(\sigma\).
\end{prop}
\begin{proof}
Note first that \(F\) is cotorsion, since \(M\) and \(\ker f\) are. Also, Lemma \ref{lem: epimorphism flat precover when ker cotorsion} shows that \(f\) is a flat precover since \(\ker f\) is cotorsion. Let \(G \to \ker f\) be the flat cover of \(\ker f\). By Lemma \ref{lem: coloc preserves short exact sequence if ker cotorsion}, the sequence
\[G^\sigma \to F^\sigma \to M^\sigma \to 0\]
is exact. After tensoring with \(\kk(\sigma)\), we get the exact sequence
\[\Topp_\sigma G \to \Topp_\sigma F \to \Topp_\sigma M \to 0.\]

Assume first that \(\Topp_\sigma f\) is not an isomorphism. The morphism \(\Topp_\sigma G \to \Topp_\sigma F\) is then non-zero. Since \(F\) is cotorsion, we have a direct summand \(R_\sigma(\mb a) \subseteq F\) by Theorem \ref{thm: flat cotorsion structure} such that the composition \(G^\sigma \to \Topp_\sigma G \to \Topp_\sigma F \to \Topp R_\sigma(\mb a) = \kk(\sigma)(\mb a)\) is non-zero. Further, we get a morphism \(R_\sigma(\mb b) \to G^\sigma\) such that the composition \(R_\sigma(\mb b) \to \kk(\sigma)(\mb a)\) is non-zero. This morphism must be an epimorphism and thus a flat cover by Remark \ref{rem: flat cover and inj hull defns agree with Miller's defn}. Now, we have the commutative diagram
\[\begin{tikzcd}
R_\sigma(\mb b) \rar & G^\sigma \rar \dar & F^\sigma \dar \\
& \Topp_\sigma G \rar & \Topp_\sigma F \rar & \kk(\sigma)(\mb a)
\end{tikzcd}\]
Since the flat cover \(R_\sigma(\mb b) \to \kk(\sigma)(\mb a)\) factors through \(F^\sigma \to \kk(\sigma)(\mb a)\), we get a morphism \(F^\sigma \to R_\sigma(\mb b)\) such that \(R_\sigma(\mb b) \to F^\sigma \to R_\sigma(\mb b)\) is an isomorphism. Thus, \(R_\sigma(\mb b)\) is a direct summand of \(F^\sigma\) with \(R_\sigma(\mb b) \subseteq \ker f^\sigma\). By Corollary \ref{cor: flat precover is cover iff no nonzero summand in kernel}, the morphism \(f^\sigma\) is not the flat cover. Lemma \ref{lem: loc and coloc preserve hulls and covers} then shows that the morphism \(f\) is not a flat cover.

Assume then that \(f\) is not the flat cover. It is still a flat precover, so \(F \cong F' \oplus R_\sigma(\mb a)\) for some \(\mb a \in \ZZ^n\) and face \(\sigma\) by Remark \ref{rem: flat precover is cover iff no R_sigma summand inside kernel}. Now, the kernel of \(\Topp_\sigma f\) contains the non-zero direct summand \(\kk(\sigma)(\mb a) = \Topp_\sigma R_\sigma(\mb a) \subseteq \Topp_\sigma F\). Hence, \(\Topp_\sigma f\) is not a monomorphism, and of course not an isomorphism.
\end{proof}

\begin{rem}
Assume that \(n > 0\). The assumptions that \(f\) is an epimorphism and \(\ker f\) is cotorsion are required in the previous proposition. For example, consider the natural embedding \(i \colon \bigoplus_{\mb a \in \ZZ^n} R(\mb a) \to \prod_{\mb a \in \ZZ^n} R(\mb a)\). Now, \(\prod_{\mb a \in \ZZ^n} R(\mb a)\) and \(\ker i = 0\) are cotorsion, \(\Topp_\sigma f \colon 0 \to 0\) for all \(\sigma \neq 0\), and
\[\Topp_0 i = \id \colon \bigoplus_{\mb a \in \ZZ^n} \kk(\mb a) \to \prod_{\mb a \in \ZZ^n} \kk(\mb a) = \bigoplus_{\mb a \in \ZZ^n} \kk(\mb a),\]
so the rest of the assumptions of the proposition hold. Still, \(i\) is clearly not an epimorphism, and of course it is not the flat cover.

Then, consider the epimorphism \(p \colon \bigoplus_{\mb a \in \ZZ^n} R(\mb a) \to \bigoplus_{\mb a \in \ZZ^n} \kk(\mb a)\). The persistence module \(\bigoplus_{\mb a \in \ZZ^n} \kk(\mb a)\) is cotorsion, and the morphisms \(\Topp_\sigma p\) are all isomorphisms. Yet \(p\) is not a flat cover since the flat cover is \(\prod_{\mb a \in \ZZ^n} R(\mb a) \to \bigoplus_{\mb a \in \ZZ^n} \kk(\mb a)\) as was shown in Example \ref{ex: flat cover of Tp}.
\end{rem}

\begin{cor}\label{cor: minimal flat resolution iff top vanishes}
Let \(M\) be a persistence module with a resolution
\[\cdots \to F_2 \xrightarrow{d_2} F_1 \xrightarrow{d_1} F_0 \xrightarrow{d_0} M \to 0,\]
where each \(F_i\) is flat and cotorsion. This resolution is the minimal flat resolution, if and only if \(\Topp_\sigma d_i = 0\) for all \(i \geq 1\) and all faces \(\sigma\).
\end{cor}
\begin{proof}
For each \(i \in \NN\), we have the cotorsion resolution \(\cdots \to F_{i+1} \to F_i \to \im d_i \to 0\) so \(\im d_i\) is cotorsion by Lemma \ref{lem: cotorsion iff cotorsion resolution}. Hence, by Lemma \ref{lem: coloc preserves short exact sequence if ker cotorsion} and the right-exactness of \(\otimes\), the sequences
\[\Topp_\sigma F_{i+1} \xrightarrow{\Topp_\sigma d_{i+1}} \Topp_\sigma F_i \to \Topp_\sigma \im d_i \to 0\]
are exact. By the previous proposition, the resolution is minimal, if and only if \(\Topp_\sigma F_i \to \Topp_\sigma \im d_i\) is an isomorphism for all faces \(\sigma\) and \(i \in \NN\). The exact sequences show that this is equivalent to \(\Topp_\sigma d_{i+1} = 0\) for all faces \(\sigma\) and \(i \in \NN\).
\end{proof}

\subsection{Duality between injective hulls and flat covers}\label{subsec: Duality between injective hulls and flat covers}

\begin{prop}\label{prop: soc top matlis duality}
Let \(\sigma\) be a face. For all persistence modules \(M\) we have a natural isomorphism
\[(\Soc_\sigma M)^\vee \cong \Topp_\sigma M^\vee.\]
\end{prop}
\begin{proof}
Using the tensor-hom adjunction, we get
\[\Topp_\sigma M^\vee = \kk(\sigma) \otimes \underline{\Hom}(R_\sigma,\underline{\Hom}(M,E(\kk)) \cong \kk(\sigma) \otimes \underline{\Hom}(M_\sigma,E(\kk)).\]
Since \(\kk(\sigma) = \kk[\sigma]_\sigma\) and \(\underline{\Hom}(M_\sigma,E(\kk))\) is an \(R_\sigma\)-module,
\[\kk(\sigma) \otimes \underline{\Hom}(M_\sigma,E(\kk)) \cong \kk[\sigma] \otimes \underline{\Hom}(M_\sigma,E(\kk)).\]
Since \(\kk[\sigma]\) is finitely presented and \(E(\kk)\) is injective,
\[\kk[\sigma] \otimes \underline{\Hom}(M_\sigma,E(\kk)) \cong \underline{\Hom}(\underline{\Hom}(\kk[\sigma],M_\sigma),E(\kk))\]
(see e.g.~the proof of \cite[Theorem 3.2.11]{enochs11}). Finally, using \(M_\sigma = (M_\sigma)^\sigma\), the tensor-hom adjunction, and \(\kk[\sigma]_\sigma = \kk(\sigma)\),
\[\underline{\Hom}(\underline{\Hom}(\kk[\sigma],M_\sigma),E(\kk)) \cong \underline{\Hom}(\underline{\Hom}(\kk(\sigma),M_\sigma),E(\kk)) = (\Soc_\sigma M)^\vee.\]
\end{proof}

\begin{thm}\label{thm: matlis duality of injective hulls and flat covers}
Let \(M\) and \(E\) be persistence modules. A morphism \(g \colon M \to E\) is the injective hull, if and only if \(g^\vee\) is the flat cover. Further, let \(F\) be a pointwise finite-dimensional persistence module with a morphism \(f \colon F \to M\). Then, \(f\) is the flat cover, if and only if \(f^\vee\) is the injective hull.
\end{thm}
\begin{proof}
We start with the morphism \(g\). Note first that \(E\) is injective and \(g\) is a monomorphism, if and only if \(E^\vee\) is flat and \(g^\vee\) is an epimorphism. Therefore, we can assume that \(E\) is injective and that \(g\) is a monomorphism. For each face \(\sigma\), the previous proposition gives us the commutative diagram
\[\begin{tikzcd}
(\Soc_\sigma E)^\vee \rar{\cong} \dar[swap]{(\Soc_\sigma g)^\vee} & \Topp_\sigma E^\vee \dar{\Topp_\sigma g^\vee} \\
(\Soc_\sigma M)^\vee \rar{\cong} & \Topp_\sigma M^\vee
\end{tikzcd}\]
Hence, \(\Topp_\sigma g^\vee\) is an isomorphism, if and only if \((\Soc_\sigma g)^\vee\) is an isomorphism, if and only if \(\Soc_\sigma g\) is an isomorphism. Since \(M^\vee\) and \(\ker g^\vee = (\ker g)^\vee\) are cotorsion and \(g^\vee\) is an epimorphism, Proposition \ref{prop: flat cover iff top isomorphisms} shows that \(g^\vee\) is the flat cover, if and only if \(g\) is the injective hull.

For the second claim, note that if \(M\) is not pointwise finite-dimensional, then \(f\) can not be an epimorphism and \(f^\vee\) can not be a monomorphism, so \(f\) is not the flat cover and \(f^\vee\) is not the injective hull. Therefore, we can assume that \(M\) is pointwise finite-dimensional. From the commutative diagram
\[\begin{tikzcd}
F \rar{\cong} \dar[swap]{f} & (F^\vee)^\vee \dar{(f^\vee)^\vee} \\
M \rar{\cong} & (M^\vee)^\vee
\end{tikzcd}\]
we see that \(f\) is the flat cover, if and only if \((f^\vee)^\vee\) is the flat cover. By the previous case, \((f^\vee)^\vee\) is the flat cover, if and only if \(f^\vee\) is the injective hull.
\end{proof}

\begin{cor}\label{cor: matlis duality of minimal resolutions}
Let \(M\) be a persistence module with a resolution \(0 \to M \to E^\bullet\). This is the minimal injective resolution, if and only if \((E^\bullet)^\vee \to M^\vee \to 0\) is the minimal flat resolution. Further, let \(F_\bullet \to M \to 0\) be a resolution, where each \(F_i\) is pointwise finite-dimensional. Then, \(F_\bullet \to M \to 0\) is the minimal flat resolution, if and only if \(0 \to M^\vee \to F_\bullet^\vee\) is the minimal injective resolution.
\end{cor}

\begin{rem}
The latter case of Theorem \ref{thm: matlis duality of injective hulls and flat covers} does not work without the assumption that \(F\) is pointwise finite-dimensional. For example, let \(M = \bigoplus_{\mb a \in \ZZ^n} \kk(\mb a)\). The flat cover of \(M\) is \(F(M) = \prod_{\mb a \in \ZZ^n} R(\mb a)\) and the injective hull is \(E(M) = \bigoplus_{\mb a \in \ZZ^n} E(\kk)(\mb a)\). Now \(F(M)^\vee\) is clearly not the injective hull, as \(E(M) \not\cong F(M)^\vee\). This can be easily seen by noting that \(E(M)_0\) has a countable basis while \((F(M)^\vee)_0\) has an uncountable basis. In short, the Matlis dual of a flat cover is not necessarily an injective hull.

However, if \(f \colon F \to M\) is a flat precover such that \(f^\vee\) is an injective hull, then \(f\) is a flat cover. To prove this, let \(h \colon F \to F\) such that \(fh = f\). Now \(h^\vee f^\vee = f^\vee\) so \(h^\vee\) must be an isomorphism as \(f^\vee\) is an injective hull. Thus \(h\) has to be an isomorphism.
\end{rem}

\begin{ex}
Let \(M\) be the persistent homology of some filtered topological space. The previous corollary shows that the minimal flat resolution of the persistent cohomology \(M^\vee\) is the Matlis dual of the minimal injective resolution of \(M\).

Assume then that \(F(M)\) is pointwise finite-dimensional. Let \(f \colon F(M) \to E(M)\) be the composition of the flat cover and injective hull of \(M\), so \(\im f = M\). The previous theorem shows that \(f^\vee \colon E(M)^\vee \to F(M)^\vee\) is then the composition of the flat cover and injective hull of \(M^\vee\). This of course also works if we switch \(M\) and \(M^\vee\) since if \(F(M^\vee)\) is pointwise finite-dimensional, then so is \(M^\vee\) and thus \(M = (M^\vee)^\vee\).

We can rephrase this in Miller's terminology \cite[Definition 5.12]{miller20b}: if the flat cover of \(M\) (resp.~\(M^\vee\)) is pointwise finite-dimensional, then the minimal flange presentation of \(M^\vee\) (resp.~\(M\)) is the Matlis dual of the minimal flange presentation of \(M\) (resp.~\(M^\vee\)). In the case of single parameter persistent homology, flange presentations are equivalent to barcodes. Thus this generalizes the well known fact that barcodes for persistent homology and cohomology are equal.
\end{ex}


\phantomsection
\addcontentsline{toc}{section}{\refname}

\bibliographystyle{alpha}
\bibliography{tex/references}


\end{document}